\documentclass[a4paper]{amsart}

\usepackage[english]{babel}
\usepackage[utf8]{inputenc}
\usepackage{amsmath, amssymb, amsthm}
\usepackage[]{tikz}
\usepackage[]{mathtools}

\newtheorem{theorem}{Theorem}
\newtheorem{definition}{Definition}
\theoremstyle{plain}
\newtheorem{corollary}{Corollary}

\newtheorem{lemma}{Lemma}
\newtheorem{proposition}{Proposition}

\newtheorem{example}{Example}

\newtheorem{remark}{Remark}

\newcommand{\cC}{{\mathcal{C}}}

\newcommand{\cS}{{\mathcal{S}}}

\newcommand{\Q}{{\mathbb{Q}}}
\newcommand{\Z}{{\mathbb{Z}}}
\newcommand{\R}{{\mathbb{R}}}
\newcommand{\bP}{{\mathbb{P}}}

\newcommand{\fp}{{\mathfrak p}}

\title[Blurred combinatorics in resolution of singularities]{Blurred combinatorics in resolution of singularities: (a little) beyond the characteristic polytope}

\author{Helena Cobo}
\email{helenacobo@gmail.com}
\author{M. J. Soto}
\address{Departamento de \'Algebra, Universidad de Sevilla}
\email{soto@us.es}
\author{Jos\'e M. Tornero}
\address{Departamento de \'Algebra \& IMUS, Universidad de Sevilla}
\email{tornero@us.es}

\date{\today}

\begin{document}

\begin{abstract}
We introduce a variation of the well-known Newton-Hironaka polytope for algebroid hypersurfaces. This combinatorial object is a perturbed version of the original one,  parametrized by a real number $\varepsilon \in \R_{\geq 0}$. For well-chosen values of the parameter, the objects obtained are very close to the original, while at the same time presenting more (hopefully interesting) information in a way which does not depend on the choice of parameter.
\end{abstract}

\maketitle

\section{Introduction}

One very powerful idea in singularity theory is Hironaka's characteristic polytope. 
The fact that it is possible to attach a finite, combinatorial object to an equation (an ideal, actually) and that it is possible to read geometric properties off such combinatorial object, led Hironaka to state that \emph{reduction of singularities is sharpening of polytopes}~\cite{H1}. 
Indeed, it is possible to see transformations of the ideal  as modifications of the polytope. 
Hironaka himself turned these ideas into his so-called {\em polyhedra games}, which were solved by Spivakovsky \cite{S1,S2} (see \cite{HS1} for a different, modern take).

Hironaka's initial research in \cite{H1} was enough for the purpose of the reduction of singularities of surfaces. However, it is also stated in~\cite{H1} that the case of bigger dimension was not fully worked out. A lot of research was carried out in the following years, both in using the device as a control tool for the resolution of singularities (see \cite{Moh}, or a more recent example in \cite{HW}) and also in studying the object for its own sake (see for instance \cite{CP,Y}).

At the same time, there is some evidence that the original definition of Newton-Hironaka polytope does not cope with certain effective problems. For instance, Piedra and the third author showed in \cite{LZ0} that it is not possible to bound the number of blowing-ups of a surface that are needed before a decrease in multiplicity, following Levi-Zariski strategy, with just the Newton-Hironaka polygon.

The core problem is that some points of the polytope may correspond to more than one monomial in the equation (see below for the precise statements). Inspired by this, we  introduce here a modification of Hironaka's definition (which we will call the {\em perturbed} polytope) and we see that this new definition gives a finer grained control of the vertices and the faces while, in some sense, keeping very close to the original one. 

The paper is structured as follows:
\begin{itemize}
\item In Section $2$ we  review the classical notion of Newton-Hironaka polytope (denoted by $\Delta (F)$).
\item In Section $3$ we   present our  version of the polytope (denoted by  $\Delta_\varepsilon (F)$).
\item In Section $4$ we  focus on the compact faces of $\Delta_\varepsilon (F)$, which are  the interesting parts of the polytope, in terms of resolution complexity~\cite{HS2}.
\item In Section $5$ we  compare the polytopes $\Delta_\varepsilon (F)$ as $\varepsilon$ varies (these are the blurred combinatorics the title refers~to).
\item In Section $6$, finally, we see how the important data for resolution purposes (mainly, the tangent cone and some permissible linear varieties) can be read in $\Delta_\varepsilon (F)$. 
\end{itemize}
We hope that the introduction of this new object may open new avenues in exploring the relationship between combinatorics and resolution of singularities.

\section{Precedents: The Newton-Hironaka polytope}
\label{S2}

For the sake of completeness, we present here some  well-known technical results that will help to understand the origin of our point of view.

Let $\cS$ be an embedded algebroid hypersurface of dimension $m$ and multiplicity $n$ over an algebraically closed field $K$ of arbitrary characteristic, and $F$ an equation of $\cS$. After a change of variables, one can take $F$ to Weierstrass form with respect to a distinguished variable $Z$:
$$
F(X_1,\ldots,X_m,Z) = Z^n + \sum_{k=0}^{n-1} a_k(X_1,\dots,X_m)Z^k,
$$
where 
$$
a_k(X_1,\ldots,X_m)  = \sum_{i_1,\ldots,i_m} a_{i_1,\ldots,i_m}^{(k)} X_1^{i_1}\cdots X_m^{i_m}\in K[[X_1,\ldots,X_m]],
$$
with  $i_1+\cdots+i_m+k \geq n$ whenever $a_{i_1,\ldots,i_m}^{(k)} \neq 0$.

One of Hironaka's great insights in his work in singularities is that we can attach a combinatorial object to $\cS$. Write
$$
N(F) = \bigl\{ (i_1,\ldots,i_m,k) \in \Z_{\geq 0}^{m+1} \; | \; a_{i_1,\ldots,i_m}^{(k)} \neq 0 \bigr\} \cup \bigl\{ (0,\ldots,0,n) \bigr\},
$$
and
$$
N^\ast(F)=N(F)\setminus\bigl\{(0,\ldots,0,n)\bigr\}.
$$

\begin{definition}
The Newton-Hironaka polytope of $F$ is
$$
\Delta (F) = \operatorname{CH} \Biggl( \bigcup_{(i_1,\ldots,i_m,k) \in N^\ast(F)} \left[  \biggl( \frac{i_1}{n-k},\ldots, \frac{i_m}{n-k} \biggr) + \R_{\geq 0}^m \right]  \Biggr) \subset \R_{\geq0}^m,
$$
where $\operatorname{CH}$ stands for the convex hull.
\end{definition}

This object appeared for the first time in the famous Bowdoin lectures \cite{Bowdoin}. Note that $\Delta (F)$ can also be read in the following way: let $\rho$ denote the mapping
\begin{eqnarray*}
\rho : N^\ast(F) & \longrightarrow & \R_{\geq0}^m \times \{ 0 \} \\
(i_1,\ldots,i_m,k) & \longmapsto & \left( \frac{i_1}{n-k},\ldots, \frac{i_m}{n-k},0 \right)
 \end{eqnarray*}

Then $\rho$ corresponds to a projection from $(0,\ldots,0,n)$ to the coordinate hyperplane $\R^m \times \{0\}$, followed by a scaling centered in $(0,\ldots,0)$ of ratio $1/n$.

We will say that a given face of $\Delta(F)$ has dimension $r$ if it is contained in an $r$-dimensional linear variety, but not in a $(r-1)$-dimensional one. Faces of dimension $0$ will be called vertices and faces of dimension $m-1$ will be called facets, as customary.

\begin{remark}
For the rest of the paper, we will identify points $(a_1,\ldots,a_m,0)\in \R^m\times\{0\}$ with points~$(a_1,\ldots,a_m)\in \R^m$. Also, if $\rho(i_1,\ldots,i_m,k)=(a_1,\ldots,a_m)$, we will say that $(a_1,\ldots,a_m)$ represents the point $(i_1,\ldots,i_m,k)$, or, abusing notation, that it represents the monomial $X_1^{i_1}\cdots X_m^{i_m}Z^{k}$. Note that this relation is not one-to-one (see below for some examples).
\end{remark}

\begin{remark}
If we allow $Z$ to vary by changes of variable of the type
$$
Z  \longmapsto Z + \alpha(X_1,\ldots,X_m), \mbox{ with } \alpha \in K[[X_1,\ldots,X_m]] \mbox{ not a unit},
$$
we obtain a collection of polyhedra which has a minimal element in the sense of inclusion. Hironaka, in~\cite{H1}, called this object the characteristic polyhedron of the pair $\bigl(\cS, \{ X_1,\ldots,X_m\}\bigr)$ and it will be denoted by $\Delta \bigl(\cS,\{X_1,\ldots,X_m\}\bigr)$.
\end{remark}

\begin{definition}
A vertex $(p_1,\ldots,p_m)$ of $\Delta (F)$ is called contractible if there exists a change of variables $\varphi$
$$
Z  \longmapsto  Z + \lambda X_1^{b_1}\cdots X_m^{b_m}, \mbox{ with } \lambda \in K,
$$
such that
$$
\Delta \bigl(\varphi(F)\bigr) \subset \Delta (F) \setminus \bigl\{(p_1,\ldots,p_m) \bigr\}.
$$
In this case, $\varphi$ is called a contraction of the vertex $(p_1,\ldots,p_m)$.
\end{definition}

It is easy to remove all contractible vertices in characteristic zero. After applying the Tchirnhausen transformation to $F$,
$$
Z \longmapsto  Z - \frac{1}{n} a_{n-1}(X_1,\ldots,X_m),
$$
the resulting equation no longer has  contractible vertices. In fact, a given vertex $(p_1,\ldots,p_m)$ is contractible if and only if it represents {\em all} the monomials of the binomial $(Z + \lambda X_1^{b_1} \cdots X_m^{b_m})^n$, and this cannot happen since $a_{n-1}(X_1,\ldots,X_m)=0$. As becomes obvious from the equations associated to the different blowing-ups, in characteristic zero this situation will persist during the resolution process. In classical terms, $Z=0$ is a hyperplane with permanent maximal contact with the hypersurface $\cS$.

Hironaka proved in \cite{H1} (for arbitrary characteristic) that the vertices of $\Delta (F)$ are not contractible if and only if $\Delta (F) = \Delta \bigl(\cS, \{X_1,\ldots,X_m\}\bigr)$. From the previous discussion this is obvious in characteristic zero, but it involves a lot of work in positive characteristic.

\begin{example}\label{nhpolygon}
The Newton-Hironaka polygon $\Delta(F)$ may not be an accurate description of the equation $F$, even if it has no contractible vertices. In fact, different points of $N^\ast(F)$ might be identified by means of $\rho$. When this phenomenon happens with the vertices of $\Delta(F)$ it can be particularly misleading, since vertices are the most important points  to keep track of, because they encode the combinatorics of the evolution of the resolution process.

Moreover, in general it is not possible to get rid of all of these {\em hidden points}, so to say, by means of changes of variables in $K[[X_1,\ldots,X_m]]$. Think, for instance, of the equation
$$
F = Z^4 + (Y-X)^4Z^2 + (Y+3X)^8.
$$
(See Figure~\ref{fig1}.)
\end{example}

\begin{figure}[htbp]
    \centering
    \begin{tikzpicture}
    \tikzset{
        every point/.style = {circle, inner sep={1.5\pgflinewidth}, 
            opacity=1, draw, solid, fill
        },
        point/.style={insert path={node[every point, #1]{}}}, point/.default={},
        point name/.style = {insert path={coordinate (#1)}},
    }
    \fill[black!20] (3.000000,0.000000)  -- (3.000000,3.000000)  -- (0.000000,3.000000)  -- (0.000000,2.000000)  -- (2.000000,0.000000)  -- cycle;
    \draw[->] (0.000000,0.000000)  -- (3.000000,0.000000)  ;
    \draw[->] (0.000000,0.000000)  -- (0.000000,3.000000)  ;
    \draw[very thin] (0,0) -- (0,-3pt) node[below] {$0$};
    \draw[very thin] (1,0) -- (1,-3pt) node[below] {$1$};
    \draw[very thin] (2,0) -- (2,-3pt) node[below] {$2$};
    \draw[very thin] (0,1) -- (-3pt,1) node[left] {$1$}; 
    \draw[very thin] (0,2) -- (-3pt,2) node[left] {$2$}; 
        \draw (0.000000,2.000000)  [point];
        \draw (0.000000,2.000000)  [point];
        \draw (0.500000,1.500000)  [point];
        \draw (0.250000,1.750000)  [point];
        \draw (1.000000,1.000000)  [point];
        \draw (0.500000,1.500000)  [point];
        \draw (1.500000,0.500000)  [point];
        \draw (0.750000,1.250000)  [point];
        \draw (2.000000,0.000000)  [point];
        \draw (1.000000,1.000000)  [point];
        \draw (1.250000,0.750000)  [point];
        \draw (1.500000,0.500000)  [point];
        \draw (1.750000,0.250000)  [point];
        \draw (2.000000,0.000000)  [point];
    \end{tikzpicture}

    \caption{Newton-Hironaka polygon of $F = Z^4 + (Y-X)^4Z^2 + (Y+3X)^8$\label{fig1}}
\end{figure}
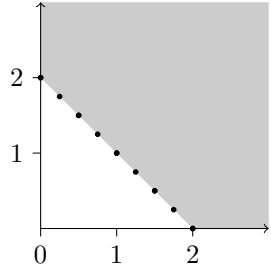

This apparent inconvenience appeared to us as the main flaw of Hironaka's characteristic polygon when it came to our original purpose of bounding the number of necessary  blowing-ups in order to drop the multiplicity. It bears repeating to finish this section that much work has been done in the study of $\Delta(F)$, from the seminal work in \cite{Bowdoin}, where the presentation might still be a bit imprecise, to more recent accounts as \cite{CP,Y}.

\section{The perturbed Newton-Hironaka polytope}
\label{Section3}

Starting from a Weierstrass equation $F$ and $N^\ast(F)$ as above, we are going to make a projection, much as Hironaka did with $\rho$, but with a small built-in perturbation. Given $\varepsilon > 0$, we define $\rho_\varepsilon$, the projection-scaling from $(0,\ldots,0,n+\varepsilon)$ with ratio $1/(n+\varepsilon)$, as:
\begin{eqnarray*}
\rho_\varepsilon : N^\ast(F) & \longrightarrow & \R_{\geq 0}^m \times \{ 0 \} \\
(i_1,\ldots,i_m,k) & \longmapsto & \biggl( \frac{i_1}{n-k+\varepsilon},\ldots, \frac{i_m}{n-k+\varepsilon},0 \biggr).
\qedhere
\end{eqnarray*}

First we will show how $\rho$ and $\rho_\varepsilon$ differ.

\begin{lemma}\label{lem.1}
Let $(i_1,\ldots,i_m,k) \in N^\ast(F)$. Then
$$
d \bigl( \rho(i_1,\ldots,i_m,k), \, \rho_\varepsilon (i_1,\ldots,i_m,k) \bigr) < \frac{\varepsilon(i_1+\cdots+i_m)}{(n-k)}.
$$
\end{lemma}

\begin{proof}
Since $\varepsilon>0$ and $n-k\geq 1$ we have
\[
d \bigl( \rho(i_1,\ldots,i_m,k), \, \rho_\varepsilon (i_1,\ldots,i_m,k) \bigr)  = \frac{\varepsilon \sqrt{i_1^2+\cdots+i_m^2}}{(n-k)(n-k+\varepsilon)} < \frac{\varepsilon(i_1+\cdots+i_m)}{(n-k)}.
\qedhere
\]
\end{proof}

\begin{remark}
If $\operatorname{char}(K)=0$, by applying the Tchirnhausen transformation  we can additionally assume that $n-k\geq 2$, and we can restate Lemma~\ref{lem.1} as
$$
d \bigl( \rho(i_1,\ldots,i_m,k), \, \rho_\varepsilon (i_1,\ldots,i_m,k) \bigr) < \frac{\varepsilon(i_1+\cdots+i_m)}{2(n-k)}.
$$
\end{remark}

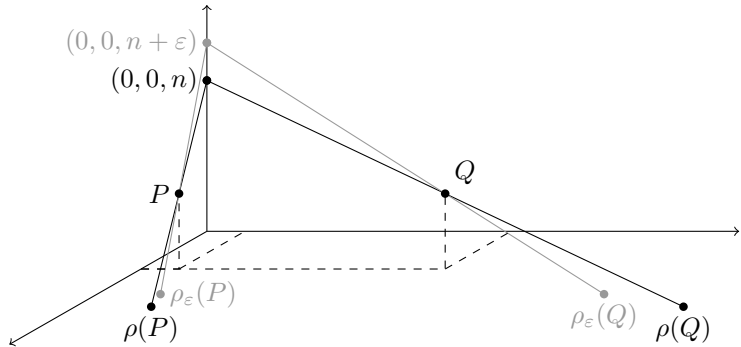
\begin{figure}[hbtp]
    \begin{tikzpicture}[]

        \begin{scope}[x = {({sin(-60)*1cm},{-cos(60)*1cm})},
                        y = {(1cm,0cm)},
                        z = {(0cm,1cm)}]
            \tikzset{
                every point/.style = {circle, inner sep={4.75\pgflinewidth}, 
                    opacity=1, draw, solid, fill
                },
                point/.style={insert path={node[every point, #1]{}}}, point/.default={},
                point name/.style = {insert path={coordinate (#1)}},
                epsi/.style={black!40}
            }
            
            \def\crad{4\pgflinewidth}
            
            \draw[->] (0,0,0)--(0,0,3);
            \draw[->] (0,0,0)--(0,7,0);
            \draw[->] (0,0,0)--(3,0,0);
            
            \draw[] (0, 0, 2.00000000000000)--(2.00000000000000, 1.00000000000000, 0) node[pos=0, left] {$(0,0,n)$} node[pos=1, below] {$\rho(P)$};
            \fill[] (0, 0, 2.00000000000000) circle (\crad);
            \fill[] (2.00000000000000, 1.00000000000000, 0)circle (\crad);
                
            \draw[epsi] (0, 0, 2.50000000000000)--(1.66666666666667, 0.833333333333333, 0) node[pos=0,left] {$(0,0,n+\varepsilon)$} node[pos=1, right] {$\rho_{\varepsilon}(P)$};
            \fill[epsi] (0, 0, 2.50000000000000) circle (\crad);
            \fill[epsi] (1.66666666666667, 0.833333333333333, 0)circle (\crad);
            
            \draw[dashed] (1,0,0) -- (1,.5,0) -- (1,.5,1) ;
            \draw[dashed] (1,.5,0) -- (0,.5,0) ;
            
            \fill (1,.5,1) circle (\crad) node[left] {$P$};
                
            \draw[] (0, 0, 2.00000000000000)--(2.00000000000000, 8.00000000000000, 0) node[pos=1, below] {$\rho(Q)$};
            \fill[] (0, 0, 2.00000000000000) circle (\crad);
            \fill[] (2.00000000000000, 8.00000000000000, 0)circle (\crad);
            \draw[epsi] (0, 0, 2.50000000000000)--(1.66666666666667, 6.66666666666667, 0) node[pos=1, below] {$\rho_{\varepsilon}(Q)$};
            \fill[epsi] (0, 0, 2.50000000000000) circle (\crad);
            \fill[epsi] (1.66666666666667, 6.66666666666667, 0)circle (\crad);
            
            \draw[dashed] (1,0,0) -- (1,4,0) -- (1,4,1);
            \draw[dashed]  (1,4,0) -- (0,4,0);
            \fill (1,4,1) circle (\crad) node[above right] {$Q$};

         \end{scope}
    \end{tikzpicture}
    \caption{The distance between $\rho(P)$ and $\rho_\varepsilon(P)$ can vary wildly.\label{fig2}}
\end{figure}

Notice that if $\varepsilon$ is not sufficiently small, this distance can be quite big (see Figure~\ref{fig2}), so that $\Delta(F)$ and $\Delta_\varepsilon(F)$ can be in fact essentially different polytopes, because we {\em lose} faces, as we see in the following example.

\begin{example}
Consider the surface defined by $F=Z^3+(X^2+XY^2)Z+X^2Y$. The classical polygon $\Delta(F)$ has two compact faces. The perturbed polygon $\Delta_\varepsilon(F)$ has two compact faces if $\varepsilon<2$ and only one compact face otherwise (see Figure~\ref{figEj} and Theorem~\ref{th1}).
\label{ejeps}
\end{example}

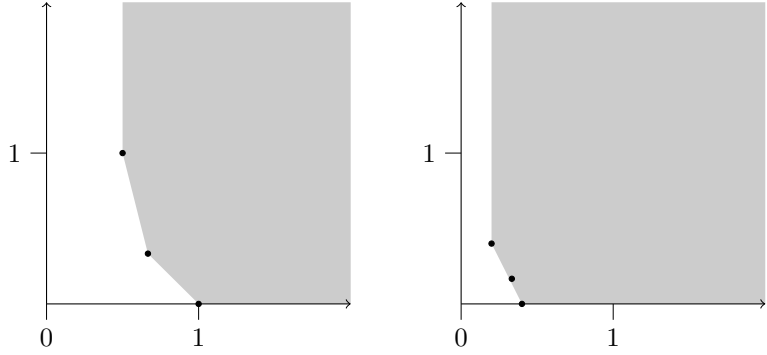
\begin{figure}[hbtp]
    \begin{tikzpicture}[scale=2]
    \tikzset{
        every point/.style = {circle, inner sep={1.75\pgflinewidth}, 
            opacity=1, draw, solid, fill
        },
        point/.style={insert path={node[every point, #1]{}}}, point/.default={},
        point name/.style = {insert path={coordinate (#1)}},
    }
    \fill[black!20] (1.000000,0.000000)  -- (2.000000,0.000000)  -- (2.000000,2.000000)  -- (0.500000,2.000000)  -- (0.500000,1.000000)  -- (0.666667,0.333333)  -- cycle;
    \draw[->] (0.000000,0.000000)  -- (2.000000,0.000000)  ;
    \draw[->] (0.000000,0.000000)  -- (0.000000,2.000000)  ;
    \draw[very thin] (0,0) -- (0,-3pt) node[below] {$0$};
    \draw[very thin] (1,0) -- (1,-3pt) node[below] {$1$};
    \draw[very thin] (0,1) -- (-3pt,1) node[left] {$1$}; 
        \draw (0.500000,1.000000)  [point];
        \draw (1.000000,0.000000)  [point];
        \draw (0.666667,0.333333)  [point];
    \end{tikzpicture}
    \qquad
    \begin{tikzpicture}[scale=2]
    \tikzset{
        every point/.style = {circle, inner sep={1.75\pgflinewidth}, 
            opacity=1, draw, solid, fill
        },
        point/.style={insert path={node[every point, #1]{}}}, point/.default={},
        point name/.style = {insert path={coordinate (#1)}},
    }
    \fill[black!20] (0.200000,0.400000)  -- (0.400000,0.000000)  -- (2.000000,0.000000)  -- (2.000000,2.000000)  -- (0.200000,2.000000)  -- cycle;
    \draw[->] (0.000000,0.000000)  -- (2.000000,0.000000)  ;
    \draw[->] (0.000000,0.000000)  -- (0.000000,2.000000)  ;
    \draw[very thin] (0,0) -- (0,-3pt) node[below] {$0$};
    \draw[very thin] (1,0) -- (1,-3pt) node[below] {$1$};
    \draw[very thin] (0,1) -- (-3pt,1) node[left] {$1$}; 
        \draw (0.200000,0.400000)  [point];
        \draw (0.400000,0.000000)  [point];
        \draw (0.333333,0.166667)  [point];
    \end{tikzpicture}
    \caption{When $\varepsilon$ is not sufficiently small, the perturbed polygon may lose faces. Left to right, the classic polygon of $F=Z^3+(X^2+XY^2)Z+X^2Y$ and the perturbed polygon of the same surface with $\varepsilon=3$.}
    \label{figEj}
\end{figure}

We would expect that if $\varepsilon$ is small enough, we should not lose information (i.e, we should not lose faces), by passing from $\Delta(F)$ to $\Delta_\varepsilon(F)$. Indeed, this is the case (see Remark \ref{Obs7}). Moreover, as is proven below, what actually happens is that we may gain new information in $\Delta_\varepsilon(F)$.

Next, we prove that $\rho_\varepsilon$ is in fact less coarse than $\rho$ in some sense.

\begin{lemma}
Let $(i_1,\ldots,i_m,k_1)$, $(j_1,\ldots,j_m,k_2)$ be different points in $N^\ast(F)$, and $\varepsilon \notin \Q$. Then
$$
d \bigl( \rho_\varepsilon(i_1,\ldots,i_m,k_1), \, \rho_\varepsilon (j_1,\ldots,j_m,k_2) \bigr) > 0.
$$
\end{lemma}

\begin{proof}
If the points are $(i_1,\ldots,i_m,k_1)$ and $(i_1,\ldots,i_m,k_2)$, the claim is obvious since
$$
\frac{i_l}{n-k_1+\varepsilon}=\frac{i_l}{n-k_2+\varepsilon} \; \Longleftrightarrow \; k_1=k_2.
$$
Otherwise, take $l$ such that $i_l\neq j_l$. If the distance was zero, then
$$
\frac{i_l}{n-k_1+\varepsilon} =  \frac{j_l}{n-k_2+\varepsilon} \Longrightarrow
\varepsilon = \frac{i_l(n-k_2)-j_l(n-k_1)}{j_l-i_l}\in\mathbb Q,
$$
which is a contradiction.
\end{proof}

\begin{remark}
From now on, taking into account the above result, we will  assume that $\varepsilon \notin \Q$. Moreover, we will take $0 < \varepsilon < 1$ in order to keep the projection $\rho_\varepsilon$ from $(0,\ldots,0,n+\varepsilon)$ close to the point $(0,\ldots,0,n)$ corresponding to the monomial $Z^n$ which gives the multiplicity. Furthermore, we will give in Lemma \ref{Lema8} another reason for this restriction on $\varepsilon$.
\label{epsCond}
\end{remark}

So, in this set-up, the projection-scaling $\rho_\varepsilon$  distinguishes all points from $N^\ast(F)$  at the cost of moving away from the natural projection~$\rho$. We now come to the main definition of this paper.

\begin{definition}
The perturbed Newton-Hironaka polytope is defined as
$$
\Delta_\varepsilon (F) = \operatorname{CH} \Biggl( \bigcup_{(i_1,\ldots,i_m,k) \in N^\ast(F)} \biggl[  \biggl( \frac{i_1}{n-k+\varepsilon},\ldots, \frac{i_m}{n-k+\varepsilon} \biggr) + \R_{\geq0}^m \biggr]  \Biggr) \subset \R_{\geq0}^m.
$$
\end{definition}

\begin{remark}
For the rest of the paper, $\Delta_0(F)$ will denote the classical Newton-Hironaka polytope (which is completely coherent with the previous definition), but whenever we use $\Delta_\varepsilon(F)$ we will assume that~$\varepsilon$  satifies the conditions of Remark~\ref{epsCond} (in particular, $\varepsilon \neq 0$).
\end{remark}

The fact that the projection $\rho_\varepsilon$ does not mix up points of $N^\ast(F)$ might be used to get a finer control of the effect of blowing ups on the polytope, but we have not pursued this research so far.

\section{Faces of $\Delta_\varepsilon (F)$}

We will turn our focus now to the faces of the perturbed polygon. Unlike the case of the original polytope $\Delta (F)$, we will see that one can distinguish between faces defined by monomials belonging to a single coefficient $a_k(X_1,\dots,X_m)$ and faces where more than one $a_k(X_1,\dots,X_m)$ need to be accounted for.

\begin{remark}
Note first that, since the vertices of $\Delta_{\varepsilon}(F)$ are in $\R_{\geq0}^m$, we may assume that any equation
$$
A_1x_1+\cdots+A_mx_m=B
$$
defining a facet, or a hyperplane containing a generic face of~$\Delta_{\varepsilon}(F)$,  has $A_i \geq 0$ and $B>0$, with no loss of generality. We will make this assumption for the rest of the paper.
Furthermore,  if we are considering a compact face, then it must hold that all~$A_i>0$.
\end{remark}

Now, by the definition of $\Delta_{\varepsilon}(F)$, its vertices cannot have rational coordinates, but we can still define a concept of rational face which will be useful in the sequel.

\begin{definition}
A face $\tau$ in $\Delta_\varepsilon(F)$, of dimension $r$ is said to be rational if it is contained in an $r$-dimensional affine space defined by
$$
A_1^{(l)}x_1+\cdots+A_m^{(l)}x_m=B^{(l)},\quad \text{for $1\leq l\leq m-r$},
$$
with $A_i^{(l)} \in \Z_{\geq0}$ and $B^{(l)} \in \mathbb R_{>0}$.
\end{definition}

Now we study the question of whether we have rational compact faces in $\partial \Delta_\varepsilon(F)$, once we choose $\varepsilon\notin\Q$.

\begin{definition}
Given $F$ and $\Delta_\varepsilon(F)$, for any face $\tau$ in $\partial\Delta_\varepsilon(F)$, we denote by $F_\tau$ the polynomial
$$
F_\tau=\sum_{\rho_\varepsilon(i_1,\ldots,i_m,k)\in\tau}a_{i_1,\ldots,i_m}^{(k)}X_1^{i_1}\cdots X_m^{i_m}Z^k
$$
\end{definition}

\begin{proposition}
A face $\tau$ (non-parallel to the coordinate hyperplanes) in $\partial \Delta_\varepsilon(F)$ is rational if and only if there exists $k$ such that
$$
F_\tau=Z^k\sum_{\rho_\varepsilon(i_1,\ldots,i_m,k)\in\tau}a_{i_1,\ldots,i_m}^{(k)}X_1^{i_1}\cdots X_m^{i_m},
$$
i.e., all the monomials of $F_\tau$ have the same exponent in $Z$.
\label{lemmaSlope}
\end{proposition}

\begin{proof}
First we will prove that given two different points in $\rho_\varepsilon(N^\ast(F))$, if they belong to a hyperplane of the form
$$
A_1x_1+\cdots+A_mx_m=B
$$
with $A_i\in\Z_{\geq 0}$ and $B\in\mathbb{R}_{>0}$, then the points must come from monomials in $F$ with the same exponent in $Z$.

Indeed, let 
$$
P_1=\left(\frac{i_1^{(1)}}{n-k^{(1)}+\varepsilon},\ldots,\frac{i_m^{(1)}}{n-k^{(1)}+\varepsilon}\right),\  P_2=\left(\frac{i_1^{(2)}}{n-k^{(2)}+\varepsilon},\ldots,\frac{i_m^{(2)}}{n-k^{(2)}+\varepsilon}\right)
$$
be two such points. Then, for $l=1,2$ we have
$$
A_1i_1^{(l)}+\cdots+A_mi_m^{(l)}=B(n-k^{(l)}+\varepsilon).
$$

Hence, denoting by $p$ and $q$ the integers $p=A_1i_1^{(1)}+\cdots+A_mi_m^{(1)}$ and $q=A_1i_1^{(2)}+\cdots+A_mi_m^{(2)}$, we deduce
$$
B=\frac{p}{n-k^{(1)}+\varepsilon}=\frac{q}{n-k^{(2)}+\varepsilon},
$$
where $p,q\neq 0$ because $B\neq 0$. Now from the last equality it follows that
$p=q$, since otherwise
$$
\varepsilon=\frac{\left( n-k^{(2)} \right)p- \left( n-k^{(1)} \right)q}{q-p}\in\Q.
$$
Therefore $k^{(1)}=k^{(2)}$.

Now, we are given $\tau$ a rational face of dimension $r$, and let $\{ P_0,...,P_r\}$ a set of points in $\rho_\varepsilon(N^\ast(F))$ which spans $L$, the $r$--dimensional linear variety containing $\tau$. Then, by definition, $L$ must be defined by a set of equations 
$$
A_1^{(l)}x_1+\cdots+A_m^{(l)}x_m=B^{(l)},\quad \text{for $1\leq l\leq m-r$},
$$
which {\em might not} be the equations of the facets intersecting in $\tau$. Nevertheless, using the result we have just proved above, the result follows.

The converse goes with similar techniques. Assume we have a face $\tau$  spanned by points representing monomials from the same coefficient $a_k(X_1,\dots,X_m)$ and let $\{P_0,\dots,P_d\}$ be a basis of the face,
$$
P_l = \biggl( \frac{i_1^{(l)}}{n-k+\varepsilon},\dots, \frac{i_m^{(l)}}{n-k+\varepsilon} \biggr), \qquad\text{ for  $l=0,\dots,d$.}
$$
Then, a system of equations for the linear variety containing $\tau$ is given by
$$
\operatorname{rank} \left( \begin{array}{cccc}
1 & x_1 & \dots & x_n \\
n-k+\varepsilon & i_1^{(0)} & \dots & i_m^{(0)} \\
\vdots & \vdots & & \vdots \\
n-k+\varepsilon & i_1^{(d)} & \dots & i_m^{(d)} \\
\end{array} \right) = d+1.
$$

As $\tau$ is spanned by $\{P_0,\dots,P_d\}$ we may assume, without loss of generality, that
$$
\left| \begin{array}{ccc}
i_1^{(0)} & \dots & i_d^{(0)} \\
\vdots & \ddots & \vdots \\
i_1^{(d)} & \dots & i_d^{(d)} \\
\end{array} \right| \neq 0.
$$

Hence, the equations of the minimal linear variety containing $\tau$ is given by the equations:
$$
\begin{cases}
\left| \begin{array}{cccc}
x_1 & \dots  & x_d & x_l \\
i_1^{(0)} & \dots  & i_d^{(0)} & i_l^{(0)} \\
\vdots & \ddots & \vdots & \vdots \\
i_1^{(d)} & \dots  & i_d^{(d)} & i_l^{(d)} \\
\end{array} \right| = 0, & \text{ for  $l=d+1,\dots ,m$,} 
\\
\left| \begin{array}{cccc}
1 & x_1 & \dots  & x_d \\
n-k+\varepsilon & i_1^{(0)} & \dots  & i_d^{(0)} \\
\vdots & \vdots & & \vdots \\
n-k+\varepsilon & i_1^{(d)} & \dots  & i_d^{(d)} \\
\end{array} \right| = 0.
\end{cases}
$$
The first $m-d$ equations are clearly homogeneous and with integer coefficients. The last equation, however, has the form
$$
\alpha_0 + \sum_{i=1}^d \alpha_i (n-k+\varepsilon) x_i = 0, \qquad \text{ with  $\alpha_0,\dots ,\alpha_d \in \Z$,}
$$
so it can be written as
$$
\alpha_1 x_1 + \dots  + \alpha_d x_d = \beta, \qquad \text{ with  $\beta \in \R$.}
$$
This proves the converse statement.
\end{proof}

Notice that, as mentioned above, one does not have this information in the classical polytope. In $\Delta_0 (F)$, all faces are rational, and hence we cannot know whether two different points in $\partial\Delta_0(F)$ correspond or not to monomials with the same exponent in $Z$.

When a given face $\tau$ is not rational we have the following additional property. 

\begin{proposition}
\label{LemIrrational}
Given a compact face $\tau$ of $\Delta_\varepsilon(F)$ which is not rational, we have
$$
\# \big( \tau\cap\rho_\varepsilon\bigl(N^\ast(F)\bigr) \big)= dim(\tau)+1.
$$
\end{proposition}

\begin{proof}
First we prove it for facets. Let $\tau$ be a non-rational compact facet of $\Delta_\varepsilon(F)$. 

By definition of the polytope, the vertices are points in $\rho_\varepsilon(N^\ast(F))$ and therefore $\# \big( \tau\cap\rho_\varepsilon\bigl(N^\ast(F)\bigr) \big)\geq m$. Let us suppose that $\tau$ contains $m+1$ points of $\rho_\varepsilon\bigl(N^\ast(F)\bigr)$
$$
P_l=\left(\frac{i_1^{(l)}}{n-k^{(l)}+\varepsilon},\ldots,\frac{i_m^{(l)}}{n-k^{(l)}+\varepsilon}\right),
$$
for $l=0,\ldots,m$. Then we deduce that
$$
\left|\begin{array}{cccc}
1 & \displaystyle \frac{i_1^{(0)}}{n-k^{(0)}+\varepsilon}  & \cdots & \displaystyle \frac{i_m^{(0)}}{n-k^{(0)}+\varepsilon}\\
\vdots & \vdots &  & \vdots \\
1 & \displaystyle \frac{i_1^{(m)}}{n-k^{(m)}+\varepsilon}  & \cdots & \displaystyle \frac{i_m^{(m)}}{n-k^{(m)}+\varepsilon}\\
\end{array}\right|=0.
$$
Hence,
$$
\left|\begin{array}{cccc}
n-k^{(0)}+\varepsilon & i_1^{(0)}  & \cdots & i_m^{(0)}\\
\vdots & \vdots &  & \vdots \\
n-k^{(m)}+\varepsilon & i_1^{(m)}  & \cdots & i_m^{(m)}\\
\end{array}\right|=0.
$$
But the expansion of  the determinant is 
$$\left|\begin{array}{cccc}
n-k^{(0)} & i_1^{(0)}  & \cdots & i_m^{(0)}\\
\vdots & \vdots &  & \vdots \\
n-k^{(m)} & i_1^{(m)}  & \cdots & i_m^{(m)}\\
\end{array}\right|+\varepsilon\left|\begin{array}{cccc}
1  & i_1^{(0)} & \cdots & i_m^{(0)}\\
\vdots & \vdots &  & \vdots \\
1  & i_1^{(m)} & \cdots & i_m^{(m)}\\
\end{array}\right|,
$$
and, since $\varepsilon\notin\mathbb Q$, we deduce the following equalities
$$
\left|\begin{array}{cccc}
n-k^{(0)} & i_1^{(0)}  & \cdots & i_m^{(0)}\\
\vdots & \vdots & & \vdots \\
n-k^{(m)} & i_1^{(m)}  & \cdots & i_m^{(m)}\\
\end{array}\right|=0    
        \quad\text{ and }\quad
\left|\begin{array}{cccc}
1  & i_1^{(0)} & \cdots & i_m^{(0)}\\
\vdots & \vdots & & \vdots \\
1  & i_1^{(m)} & \cdots & i_m^{(m)}\\
\end{array}\right|=0.
$$
Consider now the orthogonal projection
\begin{eqnarray*}
\pi: N(F) & \longrightarrow & \R_{\geq0}^m \times \{ 0 \}\\
\left( l_1,\dots ,l_m,t \right) & \longmapsto & \left( l_1,\dots ,l_m,0 \right)
\end{eqnarray*}

Then the first equality means that the points $\bigl\{ \rho \left( P_0 \right), \dots , \rho \left( P_m \right) \bigr\}$
are co-hyper\-planar in $\R^m$, while the second means that the points $\bigl\{ \pi \left( P_0 \right), \dots , \pi \left( P_m \right) \bigr\}$
are co-hyperplanar in $\R^m$. By Proposition \ref{lemmaSlope}, the integers $k^{(0)},\ldots, k^{(m)}$ are not all equal, and hence the previous two identities give a contradiction.

Hence, the $m$ points of $\tau\cap\rho_\varepsilon\bigl(N^\ast(F)\bigr)$ are on the boundary of $\tau$, and are generators of the facet. Then the result follows for faces of any dimension, since any face can be seen as intersection of facets.
\end{proof}

\begin{remark}
Proposition~\ref{LemIrrational} is no longer true if the face is rational, as the surface examples at the end of the paper show.
\end{remark}

\begin{corollary}
If $\tau$ is a non-rational compact face of $\Delta_\varepsilon(F)$ we have
$$
\operatorname{int}(\tau) \cap\rho_\varepsilon\bigl(N^\ast(F)\bigr)=\emptyset.
$$

\end{corollary}

\begin{proof}
It is a direct consequence of Proposition \ref{LemIrrational}, since the dim$(\tau)+1$ points of $\rho_\varepsilon\bigl(N^\ast(F)\bigr)$ are necessarily on the border of $\tau$.
\end{proof}

\section{Comparing perturbed polytopes}

Let $\varepsilon$,  $\varepsilon'$ be two different parameters, chosen as per Remark~\ref{epsCond}. We compare the polyhedra $\Delta_\varepsilon(F)$ and $\Delta_{\varepsilon'}(F)$.

Obviously, there is a bijective correspondence for any given pair $0<\varepsilon,\varepsilon'<1$,
\begin{eqnarray*}
T_{\varepsilon,\varepsilon'}\ \colon\ \rho_\varepsilon\bigl(N^\ast(F)\bigr) & \longrightarrow & \rho_{\varepsilon'}\bigl(N^\ast(F)\bigr)\\
P=\left(\frac{i_1}{n-k+\varepsilon},\ldots,\frac{i_m}{n-k+\varepsilon}\right)\ & \longmapsto\ & P'=\left(\frac{i_1}{n-k+\varepsilon'},\ldots,\frac{i_m}{n-k+\varepsilon'}\right)\\
\end{eqnarray*}
But we can say more when $\varepsilon$ and $\varepsilon'$ are sufficiently close, as this correspondence sends compact faces of $\Delta_\varepsilon(F)$ to compact faces of the same dimension of $\Delta_{\varepsilon'}(F)$.

Indeed, given a compact face $\tau$ of $\Delta_\varepsilon(F)$ of dimension $r$, let $P_1,\ldots,P_{r+1}$ be the points of $\rho_\varepsilon\bigl(N^\ast(F)\bigr)$ that generate $\tau$. 
Set $\tau'$ as follows
$$
\tau'=\{\lambda_1P_1'+\cdots+\lambda_{r+1}P_{r+1}'\quad |\quad \lambda_i\geq 0,\quad \lambda_1+\cdots+\lambda_{r+1}=1\}.
$$
Next we prove that, for $\varepsilon-\varepsilon'$ small, $\tau'$ is well defined and hence it is the face generated by the points $P'_1,\ldots,P'_{r+1}$.

\begin{lemma}
For $\varepsilon-\varepsilon'$ small enough and in the conditions of Remark~\ref{epsCond},  $\tau'=T_{\varepsilon,\varepsilon'}(\tau)$, and it is a compact face of the same dimension as $\tau$.
\end{lemma}

\begin{proof} By Proposition \ref{lemmaSlope}, if $\tau$ is a rational face, $\tau'$ is just the result of applying to~$\tau$ a homothety  of ratio $(n-k+\varepsilon)/(n-k+\varepsilon')$. While if $\tau$ is non-rational, we know by Proposition \ref{LemIrrational} that the points of $\tau\cap\rho_\varepsilon\bigl(N^\ast(F)\bigr)$ are exactly dim$(\tau)+1$ points, and have to be the vertices. Hence $\tau'$ is the image of $\tau$ by the correspondence $T_{\varepsilon,\varepsilon'}$.

For $\varepsilon-\varepsilon'$ small, the points $P$ and $P'$ are close enough and it is clear that $\tau$ and $\tau'$ have the same dimension.
\end{proof}

Moreover, $\tau$ is rational if and only if $\tau'$ is rational.

\begin{proposition}
For $\varepsilon-\varepsilon'$ sufficiently small, and under the conditions of Remark~\ref{epsCond}, we have that $\tau$ is a compact face of $\Delta_\varepsilon(F)$ if and only if $\tau'$ is a compact face of $\Delta_{\varepsilon'}(F)$.
\label{lematauP}
\end{proposition}

\begin{proof} Since $T_{\varepsilon,\varepsilon'}\circ T_{\varepsilon',\varepsilon}$ is the identity, we only have to prove one implication.
    Moreover, it is enough to prove the statement for facets because any face of smaller dimension must lie on the boundary of some compact facet. 

    Let then $\tau$ be a facet of $\Delta_\varepsilon(F)$. Let $P_1,\ldots,P_m\in\rho_\varepsilon\bigl(N^\ast(F)\bigr)$, 
$$
P_l=\biggl(\frac{i_1^{(l)}}{n-k^{(l)}+\varepsilon},\ldots,\frac{i_m^{(l)}}{n-k^{(l)}+\varepsilon}\biggr),\quad
\text{ for  $1\leq l\leq m$,}
$$
be $m$ vertices of $\Delta_\varepsilon(F)$ that define the hyperplane containing $\tau$. In other words, $\tau$~is contained in the hyperplane given by
$$
\left|\begin{array}{cccc}
1 & x_1 & \cdots & x_m\\
n-k^{(1)}+\varepsilon & i_1^{(1)}  & \cdots & i_m^{(1)}\\
\vdots & \vdots & & \vdots \\
n-k^{(m)}+\varepsilon & i_1^{(m)}  & \cdots & i_m^{(m)}\\
\end{array}\right|=0,
$$
where we write the equation ordering the points in such a way that $I<0$, where
$$
I=\left|\begin{array}{ccc}
i_1^{(1)} & \cdots & i_m^{(1)}\\
\vdots & & \vdots\\
i_1^{(m)} & \cdots & i_m^{(m)}\\
\end{array}\right|.
$$

The images of these points,
$$
P_l'=\frac{n-k^{(l)}+\varepsilon}{n-k^{(l)}+\varepsilon'}P_l,\quad \text{ for $1\leq l\leq m$},
$$
define the facet $\tau'$ for $\varepsilon-\varepsilon'$ sufficiently small. The defining hyperplane of $\tau'$ is given by
$$
\left|\begin{array}{cccc}
1 & x_1 & \cdots & x_m\\
n-k^{(1)}+\varepsilon' & i_1^{(1)}  & \cdots & i_m^{(1)}\\
\vdots & \vdots & & \vdots \\
n-k^{(m)}+\varepsilon' & i_1^{(m)}  & \cdots & i_m^{(m)}\\
\end{array}\right|=0.
$$
If $\tau'$ was not a facet of $\Delta_{\varepsilon'}(F)$, there must exist a vertex $P'=(p_1',\ldots,p_m')\in\rho_{\varepsilon'}\bigl(N^\ast(F)\bigr)$ such that $D'<0$, where
$$
D'=\left|\begin{array}{cccc}
1 & p_1' & \cdots & p_m'\\
n-k^{(1)}+\varepsilon' & i_1^{(1)}  & \cdots & i_m^{(1)}\\
\vdots & \vdots & & \vdots \\
n-k^{(m)}+\varepsilon' & i_1^{(m)}  & \cdots & i_m^{(m)}\\
\end{array}\right|.
$$
But since $\tau$ is a face of $\Delta_\varepsilon(F)$ we have $D>0$, where
$$
D=\left|\begin{array}{cccc}
1 & p_1 & \cdots & p_m\\
n-k^{(1)}+\varepsilon & i_1^{(1)}  & \cdots & i_m^{(1)}\\
\vdots & \vdots & & \vdots \\
n-k^{(m)}+\varepsilon & i_1^{(m)}  & \cdots & i_m^{(m)}\\
\end{array}\right|,
$$
and $P=(p_1,\ldots,p_m)$ is the image of $P'$ by the correspondence. 
For some $\widetilde{k}$ we have
$$
P=(p_1,\ldots,p_m)=\frac{n-\widetilde{k}+\varepsilon'}{n-\widetilde{k}+\varepsilon}P'\in\rho_\varepsilon\bigl(N^*(F)\bigr),
$$
and hence,
$$
D'=\frac{n-\widetilde{k}+\varepsilon}{n-\widetilde{k}+\varepsilon'}
\left|\begin{array}{cccc}
\frac{n-\widetilde{k}+\varepsilon'}{n-\widetilde{k}+\varepsilon} & p_1 & \cdots & p_m\\
n-k^{(1)}+\varepsilon' & i_1^{(1)}  & \cdots & i_m^{(1)}\\
\vdots & \vdots & & \vdots \\
n-k^{(m)}+\varepsilon' & i_1^{(m)}  & \cdots & i_m^{(m)}\\
\end{array}\right|,
$$
which, expanding the determinant along the first row, can be written as
$$
D'=I+\frac{n-\widetilde{k}+\varepsilon}{n-\widetilde{k}+\varepsilon'}
\bigl( -p_1 \det(M_1') + \cdots + (-1)^mp_m\det(M_m') \bigr),
$$
where $M_i'$ is the submatrix of $D'$ resulting from the deletion of the first row and the $(i+1)$-th column.

Now,  decompose $M_i'$ as follows
$$
M_i'=M_i^{(0)}+\varepsilon'M_i^{(1)}.
$$
Note that neither $M_i^{(0)}$ nor $M_i^{(1)}$ depend on $\varepsilon'$ or $\varepsilon$. Then,
$$
D'=I+\frac{n-\widetilde{k}+\varepsilon}{n-\widetilde{k}+\varepsilon'}(A_0+\varepsilon'A_1)
$$
where 
\begin{align*}
A_0 &= -p_1 \det \bigl( M_1^{(0)} \bigr) + \cdots + (-1)^mp_m \det \bigl( M_m^{(0)} \bigr), \\
A_1 &= -p_1 \det \bigl( M_1^{(1)} \bigr) + \cdots + (-1)^mp_m \det \bigl( M_m^{(1)} \bigr).
\end{align*}
Hence, we have
$$
D' = D + \frac{\varepsilon-\varepsilon'}{n-\widetilde{k}+\varepsilon'}A_0 +
\frac{(n-\widetilde{k})(\varepsilon'-\varepsilon)}{n-\widetilde{k}+\varepsilon'}A_1
$$
and for $|\varepsilon-\varepsilon'|$ small enough we get a contradiction since $D>0$.
\end{proof}

\begin{corollary}
In the above correspondence, if $\varepsilon,\varepsilon'$ are sufficiently small, we have  $P\in\partial\Delta_\varepsilon(F)$ if and only if $P'\in\partial\Delta_{\varepsilon'}(F)$.
\end{corollary}

Therefore we have proved that for $\varepsilon-\varepsilon'$ small enough, the polyhedra $\Delta_\varepsilon(F)$ and $\Delta_{\varepsilon'}(F)$ {\em look the same}, as the next result claims.

\begin{theorem}\label{th1}
Let $\varepsilon$, $\varepsilon'\notin\Q$, with $0<\varepsilon,\varepsilon'<1$ and $d\in\Z_{\geq 0}$. If  $\varepsilon-\varepsilon'$ is sufficiently small, the polyhedra $\Delta_\varepsilon(F)$ and $\Delta_{\varepsilon'}(F)$ have the same number of faces of dimension~$d$. Even more, the points on each face are the corresponding projection of the same point in $N^\ast(F)$.
\end{theorem}

As a consequence, polytopes $\Delta_\varepsilon(F)$ and $\Delta_{\varepsilon'}(F)$, for $\varepsilon,\varepsilon'$ small enough, though not the same, have the same information, and therefore we can talk of {\em the} perturbed Newton-Hironaka polytope, understanding that it is a family of polytopes with the same properties. We will call the {\em blurred polytope} a generic polytope of this family.

\begin{remark}
Note that if we compare the polytopes $\Delta_0(F)$ and $\Delta_\varepsilon(F)$, the proof of Proposition \ref{lematauP} also holds, and we get that, for $\varepsilon$ sufficiently small, we avoid the problem described in Example \ref{ejeps}, and every face of $\Delta_0(F)$ correspond to a face of $\Delta_\varepsilon(F)$ of the same dimension. 

Moreover, there may be new faces in $\Delta_\varepsilon(F)$, as we will see in the examples at the end of next section.
\label{Obs7}
\end{remark}

\section{Some useful information to be found in $\Delta_\varepsilon(F)$}

Now that we have compared the pertubed and the classical polytope, we proceed to search this object for useful information in terms of resolution of singularities. First we recall two important concepts.

\begin{definition}
The tangent cone of $\cS$, denoted by $\cC (\cS)$, or simply $\cC$ if no confusion arises, is the projective variety defined by
$\overline{F}$, the initial form of $F$, on $\bP^m (K)$.
\label{TC}
\end{definition}

\begin{definition}
With the above notations, let $\fp$ be a prime, non-maximal, ideal on $K[[X_1,\ldots,X_m,Z]]$ satisfying:
\begin{enumerate}
\item[(a)] $F \in \fp^n$.
\item[(b)] There are $m$ power series, $G_1,\ldots,G_m \in \fp$ such that
$\mbox{ord} (G_i) = 1$ and $\fp = (G_1,\ldots,G_m)$.
\end{enumerate}
Such a prime ideal will be called a permissible variety of $\cS$.
\end{definition}

This notion of permissible varieties agrees with the one derived from normal flatness in the work of Hironaka \cite{H2} (which is equivalent to being an equimultiple smooth subscheme).

The tangent cone is a very important invariant to keep track of during the resolution process. For a Weierstrass equation, the tangent cone is defined by $\overline{F}$, which means that, in $\Delta_0(F)$, the monomials involved correspond to points in the hyperplane  $x_1+\cdots+x_m=1$ in $\R_{\geq0}^m$. In $\Delta_\varepsilon(F)$  these points cannot be located so precisely, but they do not mix with other points, for a properly chosen $\varepsilon$, as we see in the next result.

\begin{lemma}
Let $(i_1,\ldots,i_m,k) \in N^\ast(F)$, with $i_1+\ldots+i_m+k=n$. Then, for $\varepsilon>0$, the point $\rho_\varepsilon (i_1,\ldots,i_m,k)$ lies in the region
$$
\biggl\{ 1> x_1+\cdots+x_m > 1-\varepsilon \biggr\} \subset \R_{\geq0}^m.
$$
Moreover, if $0<\varepsilon<1$, these are all the possible points of $\Delta_\varepsilon (F)$ in the above region.
\label{Lema8}
\end{lemma}

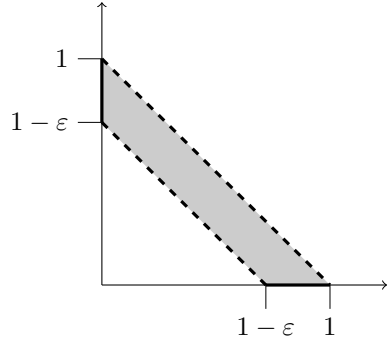
\begin{figure}[htbp]
    \centering
    \begin{tikzpicture}[scale=3]
        \def\varepsilonhalf{8pt}
        \fill[black!20] (1,0) -- (1cm-\varepsilonhalf,0)--(0,1cm-\varepsilonhalf)--(0,1) -- cycle;
        \draw[thin, ->] (0,0) -- (1.25,0);
        \draw[thin, ->] (0,0) -- (0,1.25);
        
        \draw[very thin] (1cm-\varepsilonhalf,0) -- (1cm-\varepsilonhalf,-3pt) node[below,] {$1-\varepsilon$};
        \draw[very thin] (0,1cm-\varepsilonhalf) -- (-3pt,1cm-\varepsilonhalf) node[left,] {$1-\varepsilon$};
        \draw[very thin] (1,0) -- (1,-3pt) node[below] {$1$};
        \draw[very thin] (0,1) -- (-3pt,1) node[left] {$1$};
        
        \draw[very thick] (0,1cm) -- (0,1cm-\varepsilonhalf);
        \draw[very thick] (1cm,0) -- (1cm-\varepsilonhalf,0);
        \draw[very thick, dashed]  (0,1cm-\varepsilonhalf)-- (1cm-\varepsilonhalf,0) ;
        \draw[very thick, dashed]  (0,1cm)-- (1cm,0) ;
        
        
    \end{tikzpicture}
    \caption{Every point representing a monomial  in the initial form of~$F$ lies in the shaded area.}
    \label{fig:label1}
\end{figure}

\begin{proof}
Take $(i_1,\ldots,i_m,k) \in N^\ast(F)$ with $i_1+\cdots+i_m+k=n$. Clearly
$$
\frac{i_1+\cdots+i_m}{n-k+\varepsilon} < \frac{i_1+\cdots+i_m}{n-k} = 1,
$$
and, on the other hand, since $n-k \geq 1$,
$$
\frac{i_1+\cdots+i_m}{n-k+\varepsilon} = \frac{n-k}{n-k+\varepsilon} = 1 - \frac{\varepsilon}{n-k+\varepsilon} >  1 -\varepsilon.
$$

Take now another point $(i_1,\ldots,i_m,k) \in N^\ast(F)$, with $i_1+\cdots+i_m+k>n$. It is then clear that $\rho_\varepsilon (i_1,\ldots,i_m,k)$ must be in the region $\{x_1+\cdots+x_m>1\} \subset \R^m_{\geq 0}$, provided that $\varepsilon < 1$, since
\[
\frac{i_1+\cdots+i_m}{n-k+\varepsilon} \geq \frac{n-k+1}{n-k+\varepsilon} > 1.
\qedhere
\]
\end{proof}

\begin{remark}
Over a field of characteristic zero, and after  a Tchirnhausen transformation of~$F$, we can assume that $n-k>1$. Then, the region containing the points in the tangent cone is in fact
$$
\biggl\{ 1> x_1+\cdots+x_m > 1-\frac{\varepsilon}{2} \biggr\} \subset \R_{\geq0}^m.
$$
\end{remark}

Another interesting feature of $\Delta_0(F)$ is that it is very easy to decide whether or not any linear variety defined by ideals of the form 
$$
\fp = ( Z,X_{i_1},\ldots,X_{i_r} ), \quad\text{ with }\quad i_1<i_2<\ldots <i_r, \; 1\leq r\leq m,
$$
is permissible. It is not very restrictive, since every permissible variety can be written in this form, after an appropriate change of variables. In particular, $\fp$ being permissible is equivalent to 
$$
\Delta_0 (F) \subset \{x_{i_1}+\cdots+x_{i_r} \geq 1\} \subset \R_{\geq0}^m.
$$

This still holds, suitably modified, with our setting.

\begin{lemma}\label{lemma9}
For a Weierstrass equation $F$ and $\varepsilon$ in the above conditions, the ideal $\fp = (Z,X_{i_1},\ldots,X_{i_r})$, with $1\leq r\leq m$, is permissible if and only if
$$
\Delta_\varepsilon (F) \subset \biggl\{\sum_{l=1}^r x_{i_l} \geq  1\biggr\} \bigcup \biggl\{ 1-\varepsilon < \sum_{l=1}^r x_{i_l} < 1, \; x_i=0 \mbox{ for all }i\notin\{i_1,\ldots,i_r\} \biggr\}.
$$
\end{lemma}

\begin{figure}[htbp]
    \centering
    \begin{tikzpicture}[scale=3]
        \def\varepsilonhalf{8pt}
        \fill[black!20]  (1cm,0)--(1.5,0)--(1.5,1.25)-- (1,1.25) -- cycle;
        \draw[thin, ->] (0,0) -- (1.5,0);
        \draw[thin, ->] (0,0) -- (0,1.25);
        
        \draw[very thin] (1cm-\varepsilonhalf,0) -- (1cm-\varepsilonhalf,-3pt) node[below,] {$1-\varepsilon$};
        \draw[very thin] (1,0) -- (1,-3pt) node[below] {$1$};
        \draw[black!20, very thick] (1,0)--(1cm-\varepsilonhalf,0);

    \end{tikzpicture}
    \caption{The ideal $(Z,X)$ is permissible if and only if $\rho_{\varepsilon}\bigl(N(F)\bigr)$ lies in the shaded region. Note the half open segment on the horizontal axis. }
    \label{fig:label2}
\end{figure}
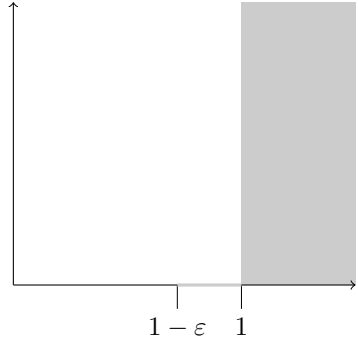

\begin{proof}
We already know that $(Z,X_{i_1},\ldots,X_{i_r})$ is permissible if and only if 
\[j_{i_1}+\cdots+j_{i_r}+k \geq n ,\]
for all $(j_1,\cdots,j_m,k) \in N^\ast(F)$. 
We have two cases:
\begin{itemize}
\item[Case 1:] $j_{i_1}+\cdots+j_{i_r}+k=n$. These points are in the tangent cone region (with $x_i=0$ for $i\in\{1,\ldots,m\}\setminus\{i_1,\ldots,i_r\}$).
\item[Case 2:] Otherwise $j_{i_1}+\cdots+j_{i_r}+k>n$. Since $0<\varepsilon<1$ and $j_l,k,n \in \Z$ this is equivalent to $j_{i_1}+\cdots+j_{i_r}+k>n+\varepsilon$.
\end{itemize}
The converse statement is also easy.
\end{proof}

\begin{remark}
If  $\operatorname{char}(K)=0$, then we can assume $n-k>1$ and the condition for $(Z,X_{i_1},\ldots,X_{i_r})$ to be permissible is
$$
\Delta_\varepsilon (F) \subset \biggl\{\sum_{l=1}^r x_{i_l} \geq  1\biggr\} \bigcup \biggl\{ 1-\frac{\varepsilon}{2} < \sum_{l=1}^r x_{i_l} < 1, \; x_i=0 \mbox{ for all }i\notin\{i_1,\ldots,i_r\} \biggr\}.
$$
\end{remark}

In previous computations in \cite{LZ0} regarding $\Delta_0 (F)$ for the case of surfaces, the most complex configuration was the so called {\em binomial segments.}

\begin{definition}
If $\cS$ is an embedded algebroid surface defined by $F$, a binomial segment of $\Delta_0 (F)$ is a segment in $\partial \Delta_0 (F)$ which contains all the points representing monomials of an expression $X^i(Y-\alpha X)^j Z^k$.
\label{binSeg}
\end{definition}

Notice that binomial segments must have slope $-1$. From Proposition \ref{lemmaSlope}, we see that binomial segments can also occur in $\Delta_\varepsilon(F)$, but there is a fundamental difference between the two cases. In fact, as shown in Example \ref{nhpolygon},  binomial segments can be immune to change of variables in $K[[X,Y]]$. However, this is not the case anymore if we replace $\rho$ by $\rho_\varepsilon$.

\begin{example}
We recall Example \ref{nhpolygon}, and compare the polygons $\Delta(F)$ and $\Delta_\varepsilon(F)$ (see Figure~\ref{f1}).

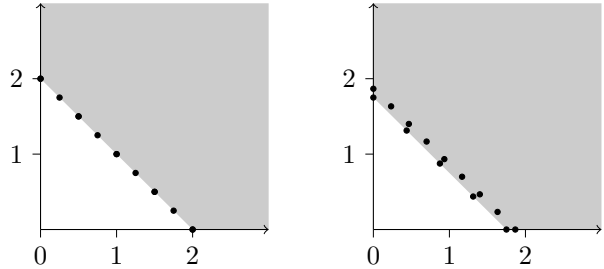
\begin{figure}[htbp]
    \centering
\begin{tikzpicture}
\tikzset{
    every point/.style = {circle, inner sep={1.75\pgflinewidth}, 
        opacity=1, draw, solid, fill
    },
    point/.style={insert path={node[every point, #1]{}}}, point/.default={},
    point name/.style = {insert path={coordinate (#1)}},
}
\fill[black!20] (3.000000,0.000000)  -- (3.000000,3.000000)  -- (0.000000,3.000000)  -- (0.000000,2.000000)  -- (2.000000,0.000000)  -- cycle;
\draw[->] (0.000000,0.000000)  -- (3.000000,0.000000)  ;
\draw[->] (0.000000,0.000000)  -- (0.000000,3.000000)  ;
\draw[very thin] (0,0) -- (0,-3pt) node[below] {$0$};
\draw[very thin] (1,0) -- (1,-3pt) node[below] {$1$};
\draw[very thin] (2,0) -- (2,-3pt) node[below] {$2$};
\draw[very thin] (0,1) -- (-3pt,1) node[left] {$1$}; 
\draw[very thin] (0,2) -- (-3pt,2) node[left] {$2$}; 
    \draw (0.000000,2.000000)  [point];
    \draw (0.000000,2.000000)  [point];
    \draw (0.500000,1.500000)  [point];
    \draw (0.250000,1.750000)  [point];
    \draw (1.000000,1.000000)  [point];
    \draw (0.500000,1.500000)  [point];
    \draw (1.500000,0.500000)  [point];
    \draw (0.750000,1.250000)  [point];
    \draw (2.000000,0.000000)  [point];
    \draw (1.000000,1.000000)  [point];
    \draw (1.250000,0.750000)  [point];
    \draw (1.500000,0.500000)  [point];
    \draw (1.750000,0.250000)  [point];
    \draw (2.000000,0.000000)  [point];
\end{tikzpicture}
\qquad
\begin{tikzpicture}
\tikzset{
    every point/.style = {circle, inner sep={1.75\pgflinewidth}, 
        opacity=1, draw, solid, fill
    },
    point/.style={insert path={node[every point, #1]{}}}, point/.default={},
    point name/.style = {insert path={coordinate (#1)}},
}
\fill[black!20] (3.000000,0.000000)  -- (1.750000,0.000000)  -- (0.000000,1.750000)  -- (0.000000,3.000000)  -- (3.000000,3.000000)  -- cycle;
\draw[->] (0.000000,0.000000)  -- (3.000000,0.000000)  ;
\draw[->] (0.000000,0.000000)  -- (0.000000,3.000000)  ;
\draw[very thin] (0,0) -- (0,-3pt) node[below] {$0$};
\draw[very thin] (1,0) -- (1,-3pt) node[below] {$1$};
\draw[very thin] (2,0) -- (2,-3pt) node[below] {$2$};
\draw[very thin] (0,1) -- (-3pt,1) node[left] {$1$}; 
\draw[very thin] (0,2) -- (-3pt,2) node[left] {$2$}; 
    \draw (0.000000,1.750000)  [point];
    \draw (0.000000,1.866667)  [point];
    \draw (0.437500,1.312500)  [point];
    \draw (0.233333,1.633333)  [point];
    \draw (0.875000,0.875000)  [point];
    \draw (0.466667,1.400000)  [point];
    \draw (1.312500,0.437500)  [point];
    \draw (0.700000,1.166667)  [point];
    \draw (1.750000,0.000000)  [point];
    \draw (0.933333,0.933333)  [point];
    \draw (1.166667,0.700000)  [point];
    \draw (1.400000,0.466667)  [point];
    \draw (1.633333,0.233333)  [point];
    \draw (1.866667,0.000000)  [point];
\end{tikzpicture}
    
    \caption{Newton-Hironaka polygons of $F = Z^4 + (Y-X)^4Z^2 + (Y+3X)^8$. Left to right, the classical and the perturbed polygons.}\label{f1}
\end{figure}

\end{example}

Actually, as a consequence of Proposition \ref{lemmaSlope}, we deduce that the border of the polytope $\Delta_\varepsilon(F)$ might not just be a homothety of $\Delta_0(F)$, as we see in the following examples. This implies that $\Delta_\varepsilon (F)$ can have more compact facets than $\Delta_0(F)$ and, in fact, one can construct easy examples in which $\Delta_0(F)$ has only one compact face, while $\Delta_\varepsilon(F)$ has as many as one wants.

\begin{example}
Consider the equation $F=Z^4+(Y^2+XY)Z^2+X^4$. We compare the polygons $\Delta(F)$ and $\Delta_\varepsilon(F)$ in  Figure~\ref{f2}. We have only one compact face in $\Delta(F)$, while in $\Delta_\varepsilon (F)$ there are two compact faces.
\label{ex1}
\end{example}

\begin{figure}[htbp]
    \centering
\begin{tikzpicture}[scale=1.5]
\tikzset{
    every point/.style = {circle, inner sep={1.75\pgflinewidth}, 
        opacity=1, draw, solid, fill
    },
    point/.style={insert path={node[every point, #1]{}}}, point/.default={},
    point name/.style = {insert path={coordinate (#1)}},
}
\fill[black!20] (2.000000,2.000000)  -- (0.000000,2.000000)  -- (0.000000,1.000000)  -- (1.000000,0.000000)  -- (2.000000,0.000000)  -- cycle;
\draw[->] (0.000000,0.000000)  -- (2.000000,0.000000)  ;
\draw[->] (0.000000,0.000000)  -- (0.000000,2.000000)  ;
\draw[very thin] (0,0) -- (0,-3pt) node[below] {$0$};
\draw[very thin] (1,0) -- (1,-3pt) node[below] {$1$};
\draw[very thin] (0,1) -- (-3pt,1) node[left] {$1$}; 
    \draw (0.000000,1.000000)  [point];
    \draw (0.500000,0.500000)  [point];
    \draw (1.000000,0.000000)  [point];
\end{tikzpicture}
\qquad
\begin{tikzpicture}[scale=1.5]
\tikzset{
    every point/.style = {circle, inner sep={1.75\pgflinewidth}, 
        opacity=1, draw, solid, fill
    },
    point/.style={insert path={node[every point, #1]{}}}, point/.default={},
    point name/.style = {insert path={coordinate (#1)}},
}
\fill[black!20] (2.000000,0.000000)  -- (0.816327,0.000000)  -- (0.344828,0.344828)  -- (0.000000,0.689655)  -- (0.000000,2.000000)  -- (2.000000,2.000000)  -- cycle;
\draw[->] (0.000000,0.000000)  -- (2.000000,0.000000)  ;
\draw[->] (0.000000,0.000000)  -- (0.000000,2.000000)  ;
\draw[very thin] (0,0) -- (0,-3pt) node[below] {$0$};
\draw[very thin] (1,0) -- (1,-3pt) node[below] {$1$};
\draw[very thin] (0,1) -- (-3pt,1) node[left] {$1$}; 
    \draw (0.000000,0.689655)  [point];
    \draw (0.344828,0.344828)  [point];
    \draw (0.816327,0.000000)  [point];
\end{tikzpicture}

    \caption{Newton-Hironaka polygons of $F=Z^4+(Y^2+XY)Z^2+X^4$.}
    \label{f2}
\end{figure}
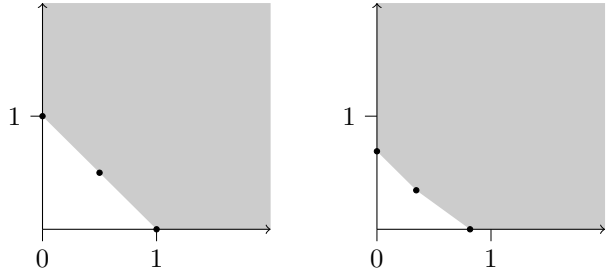

\begin{example}
Consider the equation $F=Z^8+(Y^5+XY^4)Z^3+X^5YZ^2+(X^7+X^{10})Z+Y^{10}$. The classical polygon $\Delta(F)$ consists of only one compact face, while the blurred polygon has three compact faces.
\label{ex2}
\end{example}

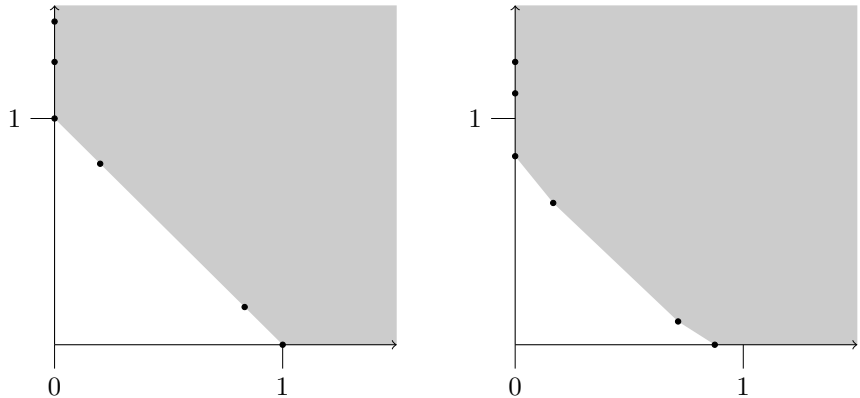
\begin{figure}[htbp]
    \begin{tikzpicture}[scale=3]
    \tikzset{
        every point/.style = {circle, inner sep={1.75\pgflinewidth}, 
            opacity=1, draw, solid, fill
        },
        point/.style={insert path={node[every point, #1]{}}}, point/.default={},
        point name/.style = {insert path={coordinate (#1)}},
    }
    \fill[black!20] (1.5,1.5)  -- (1.5000000,0.000000)  -- (1.000000,0.000000)  -- (0.000000,1.000000)  -- (0.000000,1.5)  -- cycle;
    \draw[->] (0.000000,0.000000)  -- (1.5,0)  ;
    \draw[->] (0.000000,0.000000)  -- (0,1.5)  ;
    \draw[very thin] (0,0) -- (0,-3pt) node[below] {$0$};
    \draw[very thin] (1,0) -- (1,-3pt) node[below] {$1$};
    \draw[very thin] (0,1) -- (-3pt,1) node[left] {$1$}; 
        \draw (0.000000,1.000000)  [point];
        \draw (0.000000,1.250000)  [point];
        \draw (0.000000,1.428571)  [point];
        \draw (0.200000,0.800000)  [point];
        \draw (0.833333,0.166667)  [point];
        \draw (1.000000,0.000000)  [point];
    \end{tikzpicture}
    \qquad
\begin{tikzpicture}[scale=3]
\tikzset{
    every point/.style = {circle, inner sep={1.75\pgflinewidth}, 
        opacity=1, draw, solid, fill
    },
    point/.style={insert path={node[every point, #1]{}}}, point/.default={},
    point name/.style = {insert path={coordinate (#1)}},
}
\fill[black!20] (1.500000,0.000000)  -- (0.875011,0.000000)  -- (0.714296,0.102859)  -- (0.166669,0.626678)  -- (0.000000,0.833347)  -- (0.000000,1.5)  -- (1.5,1.5)  -- cycle;
\draw[->] (0.000000,0.000000)  -- (1.5,0.000000)  ;
\draw[->] (0.000000,0.000000)  -- (0.000000,1.5)  ;
\draw[very thin] (0,0) -- (0,-3pt) node[below] {$0$};
\draw[very thin] (1,0) -- (1,-3pt) node[below] {$1$};
\draw[very thin] (0,1) -- (-3pt,1) node[left] {$1$}; 
    \draw (0.000000,0.833347)  [point];
    \draw (0.000000,1.111123)  [point];
    \draw (0.000000,1.250016)  [point];
    \draw (0.166669,0.626678)  [point];
    \draw (0.714296,0.102859)  [point];
    \draw (0.875011,0.000000)  [point];
\end{tikzpicture}
     \caption{Diagrams of $F=Z^8+(Y^5+XY^4)Z^3+X^5YZ^2+(X^7+X^{10})Z+Y^{10}$. Note that the right diagram has three distinct compact faces.}
\end{figure}

\section*{Acknowledgements}

The first author was supported by Project {\em Métodos Computacionales en Ál\-ge\-bra, D-módulos, teoría de la representación y Optimización (MTM2016-75024-P)} (Ministerio de Econom\'{\i}a y Competitividad). The second and third authors were supported by Project {\em Geometría  Aritmética, D-módulos y singularidades  (MTM\-2016–75027–P)} (Ministerio de Econom\'{\i}a y Competitividad) and Project {\em Singularidades, Geometría Algebraica Aritmética y Teoría de Representaciones: Estructuras y Métodos Diferenciales, Cohomológicos, Combinatorios y Computacionales (P12–FQM–2696)} (Jun\-ta de Andaluc\'{\i}a and FEDER).

The authors wish to express their gratitude to the referee, whose comments and suggestions helped improve the exposition and clarity of the paper.

\end{document}